\documentclass[11pt,a4paper]{article}
\usepackage{a4,amssymb,amsmath,amsthm,url,xcolor,enumerate,float}
\usepackage{bezier,amsfonts,amssymb,graphicx,amsthm,url,lineno}
\usepackage[english]{babel}
\title{Selective Hypergraph Colourings}
\date{}
\begin{document}
\newtheorem{theorem}{Theorem}[section]
\newtheorem{definition}{Definition}[section]
\newtheorem{proposition}[theorem]{Proposition}
\newtheorem{corollary}[theorem]{Corollary}
\newtheorem{lemma}[theorem]{Lemma}
\newcommand*\cartprod{\mbox{ } \Box \mbox{ }}
\newtheoremstyle{break}
  {}
  {}
  {\itshape}
  {}
  {\bfseries}
  {.}
  {\newline}
  {}

\theoremstyle{break}

\newtheorem{propskip}[theorem]{Proposition}
\DeclareGraphicsExtensions{.pdf,.png,.jpg}
\author{Yair Caro \\ Department of Mathematics\\ University of Haifa-Oranim \\ Israel \and Josef  Lauri\\ Department of Mathematics \\ University of Malta
\\ Malta \and Christina Zarb \\Department of Mathematics \\University of Malta \\Malta }

\maketitle

\begin{abstract}
We look at colourings of $r$-uniform hypergraphs, focusing our attention on unique colourability and gaps in the chromatic spectrum.  The pattern of an edge $E$ in an $r$-uniform hypergraph $H$ whose vertices are coloured is the partition of $r$ induced by the colour classes of the vertices in $E$.  Let $Q$ be a set of partitions of $r$.  A $Q$-colouring of $H$ is a colouring of its vertices such that only patterns appearing in $Q$ are allowed.  We first show that many known hypergraph colouring problems, including Ramsey theory, can be stated in the language of $Q$-colourings.  Then, using as our main tools the notions of $Q$-colourings and $\Sigma$-hypergraphs, we define and prove a result on tight colourings, which is a strengthening of the notion of unique colourability.  $\Sigma$-hypergraphs are a natural generalisation of $\sigma$-hypergraphs introduced by the first two authors in an earlier paper.  We also show that there exist $\Sigma$-hypergraphs with arbitrarily large $Q$-chromatic number and chromatic number but with bounded clique number.  Dvorak et al. have characterised those $Q$ which can lead to a hypergraph with a gap in its $Q$-spectrum.  We give a short direct proof of the necessity of their condition on $Q$.  We also prove a partial converse for the special case of $\Sigma$-hypergraphs.  Finally, we show that, for at least one family $Q$ which is known to yield hypergraphs with gaps, there exist no $\Sigma$-hypergraphs with gaps in their $Q$-spectrum.
\end{abstract}
\section{Introduction}

Let $V=\{v_1,v_2,...,v_n\}$ be a finite set, and let $E=\{E_1,E_2,...,E_m\}$ be a family of subsets of $V$.  The pair $H=(V,E)$ is called a \emph{hypergraph} with vertex- set $V(H)=V$, and with edge-set $E(H)=E$.  When all the subsets are of the same size $r$, we say that $H$ is an \emph{r-uniform hypergraph}. 

Hypergraph vertex colourings is a widely studied topic, and several different types of colourings have been defined, such as classical colourings, where monochromatic edges are not allowed \cite{berge1973graphs}, Voloshin colourings \cite{voloshin02}, non monochromatic non-rainbow (NMNR) colourings \cite{CaroLauri14} and constrained colouring \cite{bujtastuz09,ClZ1}.  In the latter three, one of the most interesting phenomenons studied is the existence of a \emph{gap} in the \emph{chromatic spectrum} of the hypergraph, which do not occur in classical colourings of neither graphs nor hypergraphs.  The \emph{chromatic spectrum} of a hypergraph $H$  is the sequence, in increasing order, of all $k$ such that $H$ has a $k$-colouring which satistifes the type of colouring being considered. We say that the chromatic spectrum has a \emph{gap} when there exist integers $k_1<k_2<k_3$ such that the hypergraph is $k_1$- and $k_3$-colourable but not $k_2$-colourable. 

In \cite{CaroLauri14} and \cite{ClZ1}, the authors have studied the existence, or not, of gaps in the chromatic spectrum for NMNR colourings and constrained colourings in particular, using a construction defined in \cite{CaroLauri14}, the $\sigma$-hypergraph.  A $ \sigma$-hypergraph $ H(n,r,q$ $\mid$ $\sigma$), where $\sigma$ is a partition of $r$,  is an r-uniform hypergraph having $nq$ vertices partitioned into $ n$ \emph{classes} of $q$ vertices each.  If the classes are denoted by $V_1$, $V_2$,...,$V_n$, then a subset $K$ of $V(H)$ of size $r$ is an edge if the partition of $r$ formed by the non-zero cardinalities $ \mid$ $K$ $\cap$ $V_i$ $\mid$, $ 1 \leq i \leq n$, is $\sigma$. The non-empty intersections $K$ $\cap$ $V_i$ are called the parts of $K$.  The number of parts of the partition $\sigma$ is denoted by $s(\sigma)$, while the size of the largest and smallest parts of the partition $\sigma$ are denoted by $\Delta=\Delta(\sigma)$ and $\delta=\delta(\sigma)$, respectively.  Several interesting results emerged about the chromatic spectrum of $\sigma$-hypergraphs in conjecunction with these two types of colourings

For the purpose of this paper we now extend this definition of a $\sigma$-hypergraph to the more general notion of a $\Sigma$-hypergraph.  Let $P(r)$ denote the set of all partitions of $r$. A $\Sigma$-hypergraph $ H(n,r,q$ $\mid$ $\Sigma$), where $\Sigma \subseteq P(r)$,  is an $r$-uniform hypergraph having $nq$ vertices partitioned into $ n$ \emph{classes} of $q$ vertices each.  If the classes are denoted by $V_1$, $V_2$,...,$V_n$, then a subset $K$ of $V(H)$ of size $r$ is an edge  if the partition $\sigma$ of $r$ formed by the non-zero cardinalities $ \mid$ $K$ $\cap$ $V_i$ $\mid$, $ 1 \leq i \leq n$, is in $\Sigma$.  We call $\sigma$ the \emph{edge type} of the edge $E$, denoted by $type(E)$, and $\Sigma$ is the set of allowable edge-types in $H(n,r,q | \Sigma)$.

In this paper we describe \emph{$L$-colourings} and \emph{$Q$-colourings} of $r$-uniform hypergraphs, which can encompass many different types of hypergraph colourings described in the literature.  Such ideas, originated in Voloshin's seminal work in \cite{jiang2002chromatic}, were more explicit in \cite{Milici},  in the context of the colouring of Steiner systems for  $r=3$ and $r=4$, and further studied in \cite{griggs2008some,quattrocchicol}, and studied in much more generality for oriented graphs in \cite{2010pattern}.  We define $L$- and $Q$-colourings as follows.

Let $H$ be an $r$-uniform  hypergraph, $ r \geq 2$ and consider $E \in E(H)$. Then a colouring of the vertices of $E$ induces a partition  $\pi$ of $r$ whose parts are the numbers of vertices of each colour appearing in $E$.  This partition is called the \emph{colour pattern} of $E$ and is written as $pat(E)  = (n_1,n_2,\ldots,n_k)$, where $n_1 \geq n_2 \geq \ldots \geq n_k \geq 1$, where $\sum_{i=1}^{k} n_i = r$.

For any  edge $E \in E(H)$, we assign $Q(E) \subseteq P(r)$.  We say that  a colouring of the vertices of $H$ is an \emph{$L$-colouring}, where $L=\{Q(E_i): i=1, \ldots, |E(H)|\}$, if $\forall E_i \in E(H)$, $pat(E_i) \in Q(E_i)$.  In the case when all the edges are assigned the same family of partitions $Q$, that is $Q(E_i)=Q$, $\forall E_i \in E(H)$,we call this a $Q$-colouring. We define $\Delta(Q)= \max \{ \Delta(\pi): \pi \in Q \}$ and let $s(Q)= \max \{ s(\pi): \pi \in Q\}$. 

The \emph{$Q$-spectrum} of $H$ is the sequence, in increasing order, of all $k$ such that $H$ has a $k$-$Q$-colouring.   The \emph{lower chromatic number} $\chi_{Q}$ is defined as the least number $k$ for which $H$ has a $k$-$Q$-colouring.  Similarly, the \emph{upper chromatic number} $\overline{\chi}_Q$ is the largest $k$ for which $H$ has a $k$-$Q$-colouring.   Clearly, the first and last terms of this sequence are $\chi_Q$ and $\overline{\chi}_Q$ respectively.  As described previously, we say that the $Q$-spectrum has a \emph{gap}, or is \emph{broken}, when there exist integers $k_1<k_2<k_3$ such that the hypergraph is $k_1$-$Q$- and $k_3$-$Q$-colourable but not $k_2$-$Q$-colourable.  We denote the $Q$-spectrum of a hypergraph $H$ by $Spec_Q(H)$ . 

This paper is structured as follows.  We first use the ``language" of $L$- and $Q$-colourings to describe the different types of colourings covered by this concept. In Section 2 we introduce, motivated by the notion of uniquely-colourable graphs and hypergraphs \cite {Tuza2002221}, the concept of \emph{tight Q-colourings} of hypergraphs and give some interesting results involving $\Sigma$-hypergraphs.   In Section 3, we then consider  the clique number of $\Sigma$-hypergraphs  and present sparse $r$-uniform $\Sigma$-hypergraphs with bounded clique number and arbitrary large chromatic number ,extending results from\cite{ClZ2}.  Finally, in Section 4 we concentrate on conditions on $Q$ for the existence and non-existence of gaps  in the $Q$-spectrum of $r$-uniform hypergraphs.  These conditions have been completely characterised in \cite{2010pattern} for oriented hypergraphs, with the results for non-oriented hypergraphs given as a corollary in Theorem 20  in \cite{2010pattern}. Here we show that, in many cases, these conditions also guarantee the existence of a $\Sigma$-hypergraph with a broken $Q$-spectrum.  But, we also show that $\Sigma$-hypergraphs cannot always provide examples with a broken $Q$-spectrum for all cases of $Q$ for which it is known that a uniform hypergraph with a gap in its $Q$-spectrum does indeed exist.

Several types of colourings have been defined for hypergraphs.  Here we look at the main types which have been studied, and describe them in terms of $L$- and $Q$-colourings.  We also look at extermal colouring problems such as Ramsey Theory, and describe them also in terms of $Q$-colourings.

Two particularly important partitions of $r$ will be used several times --- the monochromatic partition $(r)$ and the rainbow partition $(1,1, \ldots, 1)$.  We use $M$ and $R$ to represent these partitions respectively.

\begin{enumerate}
\item{For classical  graph colourings, $Q =\{\pi \in P(r): \pi=R\}$.}
\item{For, classical colourings of hypergraphs,   $Q = \{\pi \in P( r): \pi \not = M \}$.  The chromatic number $\chi(H)$ is the smallest integer $k$ for whicha proper classical $k$-colouring of $H$ exists.  For any integer $p \geq \chi(H)$, there exists a proper classical $p$-colouring.}
\item{ In Voloshin colourings of hypergraphs \cite{voloshin02}, or mixed hypergraphs, there exist two types of edges, $D$-edges and $C$-edges.  A $D$-edge cannot be \emph{monochromatic}, that is all vertices of the edge having the same colour, while a $C$-edge cannot by \emph{polychromatic(rainbow)}, that is all vertices having a different colour.  Hence for all  $D$-edges, $Q_D=Q(E)=\{\pi \in P( r): \pi \not =M\}$, while for all $C$-edges, $Q_C=Q(E)=\{\pi \in P( r): \pi \not =R\}$, and $L=\{Q_C \cup Q_D\}$}
\item{A non-monochromatic non-rainbow (NMNR) colouring, as discussed in \cite{CaroLauri14},  is a $Q$-colouring where $Q= \{\pi \in P( r): \pi \not \in \{ M,R \}\}$.  Such hypergraphs are often referred to as \emph{bi-hypergraphs} and are a special type of Voloshin colourings.}
\item{An  $(\alpha,\beta)$-colouring of a hypergraph $H$, as described in \cite{Clz1}, is the case where $Q  = \{\pi \in  P(r):  \pi= (n_1 ,n_2 ,\ldots ,n_k ),  \alpha \leq k \leq \beta\}$ .  This is based on the concept of \emph{colour-bounded hypergraphs} first defined by Bujtas and Tuza in \cite{bujtastuz09}.  Observe that classical hypergraph colourings are $(2,r)$-colourings, while NMNR-colourings are $(2,r-1)$-colourings.}
\item{Bujtas and Tuza defined another type of hypergraph colouring with further restrictions in \cite{bujtas2010color}:  a \emph{stably bounded hypergraph} is a hypergraph together with four colour-bound functions which express restrictions on
vertex colorings.  Formally,  an $r$-uniform stably bounded hypergraph is a six-tuple $H = (V(H),E(H), s, t, a, b)$, where $s$,$t$, $a$ and $b$ are positive integers,  called \emph{colour-bound functions}. We assume throughout that  $1 \leq s \leq t \leq r$ and $1 \leq a \leq b \leq r$ hold.  A proper vertex coloring of $H = (V(H),E(H), s, t, a, b)$   satisfies the following three conditions for every edge $E \in E(H)$.
\begin{itemize}
\item{ The number of different colors assigned to the vertices of $E$ is at least $s$ and at most $t$.}
\item{There exists a colour assigned to at least $a$ vertices of $E$.}
\item{Each color is assigned to at most $b$ vertices of $E$.}
\end{itemize}

Hence, such a colouring is a $Q$-colouring where  \[Q  = \{\pi \in  P(r): \pi= (n_1 ,n_2 ,\ldots ,n_k ), n_1 \geq n_2 \geq \ldots \geq n_k,   s \leq k \leq t,a \leq  n_1  \leq b .\}\]}
\item{In \cite{axerol14}, given a vertex coloring of a hypergraph, a vertex contained in an edge $E$ is said to be uniquely colored in $E$, if its color is assigned to no other vertex in $E$. If every edge of a hypergraph contains a uniquely colored vertex, then the coloring is called \emph{conflict-free}  In this case, $Q  = \{\pi \in  P(r): \pi= (n_1 ,n_2 ,\ldots ,n_k ), n_1 \geq n_2 \geq \ldots \geq n_k=1.\}$}
\item{We can also use the language of $Q$-colourings to express Ramsey's Theorem, which broadly states that one will find monochromatic cliques in any edge colouring of a sufficiently large complete graph. The theorem can also be extended to hypergraphs:  for fixed integers $r$ and $k$, and integers $n_1\ldots,n_k \geq r$, there exists an integer $R(r,n_1, \ldots, n_k)$ such that if the edges of a complete $r$-uniform hypergraph on $n$ vertices,  $n \geq R(r,n_1, \ldots, n_k)$ are coloured using $k$  colours, then for some $1 \leq j \leq k$, there is a monochromatic complete $r$-uniform subhypergraph of order $n_j$ and colour $j$.

Let us consider $K(n,r)$ to be a complete $r$-uniform hypergraph on $n$ vertices.  We define the hypergraph $H=H(n,r,p)$, where $p \geq r+1$, such that the vertices of $H$ are the edges of $K(n,r)$, so that $|V(H)|=\binom{n}{r}$. Now a collection of $\binom{p}{r}$ vertices in $H$ is an edge if the corresponding edges in $K(n,r)$ are all taken from a set of $p$ vertices in $K(n,r)$.  So $|E(H)|= \binom{n}{p}$ and $H$ is  $\binom{p}{r}$-uniform.  We use this to express Ramsey's Theorem as follows:

\begin{theorem}
\noindent For fixed $k,r,p$, there exists $N(k,r,p)$ such that for $n \geq N(k,r,p)$, if $Q \subseteq P(\binom{p}{r})$ and $M \not \in Q$, then $H=H(n,r,p)$ is not $k$-$Q$-colourable.
\end{theorem}}
\end{enumerate}

The definitions of $L$-colourings and $Q$-colourings encompasses all these different types of colourings as special cases, and allows for other types of colourings to be defined.

\section{Tight colourings}

 We say that a hypergraph $H$ is \emph{tightly $Q$-colourable} if:

\begin{enumerate}
\item{$Spec_Q(H)=\{k\}$}
\item{The $k$-$Q$-colouring is unique (up to exchanging colour classes)}
\item{The colour classes have equal size.}
\item{For any partition $\pi \in Q$, $H$ is not $(Q \backslash \pi)$-colourable.}
\end{enumerate}

Tight colouring is motivated by the concept of the so called uniquely colourable graphs \cite{gross2004handbook}.   While in a  uniquely colourable graph, uniqueness is required only when the number of colours used is exactly $\chi(G)$  (but if we can  use more than $\chi(G)$ colours, there is no restrictions on the colour class structure), in tight-$Q$ colouring of hypergraphs, the restrictions are considerably stricter.   

 An important colouring that will be used frequently is the \emph{canonical distinct monochromatic colouring} --- in short CDMC --- which  is a colouring of the $\Sigma$-hypergraph $H(n,r,q|\Sigma)$ such that all vertices belonging to the same class are given the same colour (so each class is monochromatic), and vertices of distinct classes are given distinct colours (hence distinct).  It is a simple but important fact that in a CDMC  of $H(n,r,q|\Sigma)$,  we have  $type(E ) = pat(E)$ for every edge $E \in E(H)$.

We are now ready to prove our main result on tight colouring.

\begin{theorem} \label{tight}
Suppose $Q \subset P(r)$ such that $M,R \not \in Q$.  Then the $\Sigma$-hypergraph $H=H(2r,r,(r-1)^2+1 \mid \Sigma=Q)$ (on $(2r((r-1)^2+1)$ vertices) is tightly $Q$-colourable.

\end{theorem}

\begin{proof}

Consider $H=H(2r,r,(r-1)^2+1 \mid \Sigma=Q)$.  Clearly $H$ is $2r$-$Q$-colourable by the CMDC, with colour classes of equal size.

We first show that $Spec_Q(H)=\{2r\}$, and that $H$ is uniquely $2r$-$Q$-colourable.

Consider any $k$-$Q$-colouring of $H$.  Let $n_j$ be the number of distinct colours appearing in  class $V_j$, for $j=1 \ldots 2r$.  Let us consider those classes for which $n_j \geq r$.  If there are at least $r$ such classes, then we can choose any edge $E$ of $H$ with $type(E )  = \sigma \in \Sigma $ to include $r$ distinct colours, so that $pat(E ) = R \not \in Q$.

So there are at most $r-1$ such classes, and hence at least $r+1$ classes in which $n_j \leq r-1$.  Let these classes be $V_1$ to $V_{r+1}$.  Since $q=(r-1)^2+1$, it follows by the pigeon-hole principle that in each of the classes $V_1$ to $V_{r+1}$ there is a colour that appears at least $r$ times.  Let $c_j$ be the colour which appears most frequently in class $V_j$ for $j=1 \ldots r+1$.  

We show that the colours $c_1,\ldots c_{r+1}$ must all be distinct and cannot appear in any other class.  Suppose that colour $c_j$ appears in another class among $V_1$ to $V_{2r}$, say in class $V_t$.  Let $\sigma$ be a partition such that $\Delta(\sigma)=\Delta(Q)$.  Since $M \not \in Q$, $\Delta(\sigma)=\Delta(Q) <r$.  Let $E$ be an edge with $type(E) =\sigma$.  Let us choose the vertices of the $\Delta$-part of $E$ from the class $V_j$ all of colour $c_j$, and we choose another part of  $E$  from class $V_t$ to include the vertex with colour $c_j$, with the remaining parts chosen from any of the remaining classes.  Then $pat(E)$ has $\Delta+1$ vertices of the same colour, which is impossible by the maximality of $\Delta(Q)$.  Hence $c_1,\ldots,c_{r+1}$ are all distinct and the colour $c_j$ only appears in class $V_j$ for $j=1,\ldots,r+1$.  

We now show that in fact, classes $V_1$ to $V_{r+1}$ must be monochromatic of distinct colours.  Suppose some class $V_j$, $1 \leq j \leq r+1$, has two vertices $x$ and $y$ of distinct colours, one of which, say $x$, has colour $c_j$.  Let $\sigma$ be a paritition such that $s(\sigma)=s(Q)$.  Since $R \not \in Q$, $s(\sigma)=s(Q) =s \leq r-1$, and hence $\Delta(\sigma) \geq 2$.  Also $s \geq 2$ since $M \not \in Q$.  We now take an edge $E$ with $type(E)=\sigma$ such that the $\Delta$ part of $E$ is chosen from $V_j$ to include $x$ and $y$.  In the remaining classes from $V_1$ to $V_{r+1}$, there are at least another $r-1$ distinct colours among the $c_1$ to $c_{r+1}$.  So we can choose the remaining parts of $\sigma$  with a distinct colour for each differnt to the colours of $x$ and $y$.  But then this edge includes $s(Q)+1$ colours which is not possible by the maximality of $s(Q)$.

Finally, we show that all the class $V_1$ to $V_{2r}$ must be monochromatic of distinct colours.  Suppose $V_j$ for $j > r+1$, has two vertices, $x$ and $y$ of distinct colours.  Observe that these colours cannot appear in classes $V_1$ to $V_{r+1}$ as shown above.  Let $\sigma$ again be a partition such that $s(\sigma)=s(Q)$.  Recall that, since $R \not \in Q$, $s(\sigma)=s(Q) =s \leq r-1$, and hence $\Delta(\sigma) \geq 2$, and also $s \geq 2$ since $M \not \in Q$.  We take an edge $E$ with $type(E)=\sigma$ such that the $\Delta$-part is taken from $V_j$ to include $x$ and $y$, and the remaining parts are chosen from the classes $V_1$ to $V_{s-1}$.  Again $E$ has a colour pattern with at least $s(Q)+1$ distinct colours, a contradiction.  Hence all classes must be monochromatic of distinct colour, and, up to permutation of classes,  the CMDC is the only possible colouring of $H$.

\medskip

We now show that for any partition $\pi \in Q$, $H$ is not $(Q \backslash \pi)$-colourable.  Clearly if $|Q|=1$, the result is trivial, so we may assume that $|Q| \geq 2$.  Now consider $\pi \in Q$ and $Q^* = Q \backslash \pi$.  Clearly the CMDC is not a valid $(2r)$-$Q^*$-colouring of $H$, since if $E \in E(H)$ has $type(E)=\pi$, then also $pat(E)=\pi \in \Sigma=Q$, but $pat(E)=\pi \not \in Q^*$.

So let us assume any $Q^*$-colouring of $H$ which is not the CDMC.  Once again let $n_j$ be the number of colours appearing in $V_j$ for $j=1 \ldots 2r$.  If there are at least $r$ classes for which $n_j \geq r$, then we can choose any edge $E$ of $H$ with $type(E)=\sigma \in \Sigma$ to include $r$ distinct colours, so that $pat(E)=R \not \in Q$ and hence not in $Q^*$.  As before, there are at most $r-1$ such classes, and hence at least $r+1$ classes in which $n_j \leq r-1$.  Let these classes be $V_1$ to $V_{r+1}$.  Since $q=(r-1)^2+1$, it follows by the pigeon-hole principle that in each of the classes $V_1$ to $V_{r+1}$ there is a colour that appears at least $r$ times.  Let $c_j$ be the colour which appears most frequently in class $V_j$ for $j=1 \ldots r+1$.  We show that the colours $c_1,\ldots c_{r+1}$ must all be distinct and cannot appear in any other class.  Suppose that colour $c_j$ appears in another class among $V_1$ to $V_{2r}$, say in class $V_t$.  Let $\sigma$ be a parition such that $\Delta(\sigma)=\Delta(Q)$ and observe that $\Delta(Q) \geq \Delta(Q^*)$.  Also since $M \not \in Q$, $\Delta(Q)=\Delta(\sigma)<r$.

Let $E$ be an edge with $type(E)=\sigma$.  We choose the $\Delta$-part of $E$ from the vertices in $V_j$ of colour $c_j$, and another part from class $V_t$ to include the vertex of colour $c_j$, with the rest of the parts chosen from the remaining classes.  Then $pat(E)$ has at least $\Delta(Q)+1 \geq \Delta(Q^*)+1$ vertices of the same colour $c_j$, which is not valid in a $Q^*$-colouring of $H$.  Hence $c_1,\ldots ,c_{r+1}$ are all distinct and each colour $c_j$ only appears in class $V_j$ for $j=1 \ldots r+1$.  But now let us choose an edge $E$ with $type(E)=\pi \in \Sigma$ from among the classes $V_1,\ldots, V_{r+1}$.  This edge has $pat(E)=\pi \not \in Q^*$, making the $Q^*$-colouring of $H$ invalid.  Hence $H$ is not $(Q \backslash \pi)$-colourable.

\end{proof}

\section{Bounded clique number and Unbounded Chromatic Number}

We now consider the size of the largest clique, $\omega(H)$, for a $\Sigma$-hypergraph $H$.   In general finding the size of the largest clique in a graph or hypergraph is $NP$-complete problem \cite{garey2002computers}.  Here we will show that for $\Sigma$-hypergraphs $H(n,r,q|\Sigma$ ), $\omega(H)$ depends only on some structural properties of $\Sigma$,  and hence $\omega(H$) can by computed in $O(1)$ time for fixed $\Sigma$.

We start with a definition.  For a family $F \subseteq \Sigma$ of partitions of $r$ and $k \geq r$, we say that $F$ is \emph{$k$-full} if the following conditions hold:
\begin{enumerate}
\item{There exist positive integers $b_1,b_2,\ldots, b_t$ such that $b_1+b_2+\ldots+b_t=k$.}
\item{For every $A=\{a_1,\ldots,a_t\}$ such that $b_i  \geq a_i \geq 0$ for $i=1 \ldots t$ and $\sum_{i=1}^{t}a_i = r$, there exists $\pi \in F$ such that the elements of $\pi$ are $sup (A)=\{a_i: a_i>0\}$, the \emph{support} of $A$.}
\end{enumerate}

\begin{theorem}
Given $H=H(n,r,q \mid \Sigma)$ with $n \geq s(\Sigma)$ and $q \geq \Delta(\Sigma)$, then $\omega(H) = \max \{k: \exists F \subseteq \Sigma \mbox{ where $F$ is $k$-full}\}$.
\end{theorem}

\begin{proof}

Let $K$ be a fixed maximal clique.  Suppose $\{v_1,\ldots,v_r\}$ is a set of $r$ vertices in $K$.  Since $K$ is a clique, these vertices must form an edge $E$ with $type(E)=\pi \in \Sigma$.  Clearly $E = \bigcup\{E \cap V_i: i=1 \ldots n\}$.

Let $a_i=|E \cap V_i|$ and let $b_i = |K \cap V_i|$.  We observe that $\sum_1^n b_i=|K|$ and $\sum_i^n a_i=r$, and $a_i \leq b_i$ for $i=1 \ldots n$.  Hence we define $F=\{ \pi : \exists E \subset K \mbox{ with } type(E)=\pi\}$, and we see that $F \subset \Sigma$ is $k$-full, proving that $\omega(H) \leq max \{ k: \exists F \subset \Sigma \mbox{ where $F$ is $k$-full}\}$.

Now suppose $F \subset \Sigma$ is $k$-full.  Then there exist $b_1,\ldots, b_t$ such that $b_1+b_2+\ldots+b_t=k$..  Let us assume without loss of generality that $b_1 \geq b_2 \geq \ldots \geq b_t$.

Let $K$ be a set of vertices obtained by choosing $b_1$ vertices from $V_i$ for $i=1 \ldots t$.  We show that $K$ is a clique.

Let $E=\{v_1,\ldots,v_r\}$ be a set of vertices in $K$.  Let $A=\{|E \cap V_i|=a_i, i=1 \ldots t\}$.  We observe that $a_1+a_2+\ldots+a_t=r$, and $a_i \leq b_i$. 

Since $F$ is $k$-full, there is a partition $\pi \in F$ whose parts are the elements of $sup(A)$.  Hence $E$ is an edge in $K$, and since $E$ was arbitrarily chosen,, $K$ is a clique, proving that $\omega(H) \geq max \{ k: \exists F \subset \Sigma \mbox{ where $F$ is $k$-full}\}$.

\end{proof}

The next corollary shows that unless $M$ or $R$ are in $\Sigma$,  $H(n,r,q|\Sigma)$ has bounded clique size.

\begin{corollary} \label{omega}
Consider $H=H(n,r,q \mid \Sigma)$ with $M,R \not \in \Sigma$.  Then $\omega(H) \leq (r-1)^2$.
\end{corollary}

\begin{proof}

Let $K$ be a maximal clique in $H$.  Suppose $K$ intersects at lest $r$ classes $V_i$.  Then if we take a single vertex from $r$ different classes, this must form an edge$E$  since $K$ is a clique, but $type(E)=R \not \in \Sigma$.  Hence $K$ intersects at most $r-1$ classes.

Now consider the intersection of $K$ with a single class $V_i$.  Suppose $|K \cap V_i| \geq r$.  Then if we choose $r$ vertices from $K \cap V_i$, this must form an edge $E$ since $K$ is a cliqu, but $type(E)=M \not \in \Sigma$.  Hence the size of the intersection of $K$ with any class must be less than $r$.  Therefore $\omega(H) \leq (r-1)^2$.
\end{proof}

We now consider the $Q$-chromatic number $\chi_Q(H)$ and the chromatic number $\chi(H)$, and show that we can construct $\Sigma$-hypergraphs with arbitrarily large chromatic number, but with bounded clique number, a subject that got much attention \cite{erdos1975problems,Gebauer20131483,kostochka2010constructions}.

\begin{theorem}

Consider $H=H(n,r,q \mid \Sigma)$ with $M,R \not \in \Sigma$ and $M \not \in Q \subseteq P(r)$.  Let $t$ be an arbitrarily large integer.  Then there exists a $\Sigma$-hypergraph $H$ with $\chi_Q(H) \geq \chi(H) \geq t$ and  $\omega(H) \leq (r-1)^2$.
\end{theorem}

\begin{proof}

Let $H = H(n,r,q \mid \Sigma)$.  Let $n \geq (r-2)(t-1)+1$ and $q=(r-2)(t-1)+1$.  By Theorem \ref{omega}, $\omega(H) \leq (r-1)^2$, and $H$ is $n$-$Q$-colourable using the CDMC.  This is a valid classical colouring since $M \not \in \Sigma$, and hence $\chi_Q(H) \geq \chi(H)$.

We now show that $H$ is not $k$-colourable, for $k < t$.  Consider a colouring using $k$ colours, where $ 1 < k <t$.  Since $M, R \not \in \Sigma$, it follows that $\Delta(\Sigma) \leq r-1$ and $s(\Sigma) \leq r-1$.

Since $q\geq (r-2)(t-1) + 1$ and we use at most $t-1$ colours, in each class there must be a colour which appears at least $r-1$ times.  Since $n \geq (r-2)(t-1)+1$ and we use at most $t-1$ colours, it follows that one of the colours appears at least $r-1$ times in $r-1$ classes.  Thus there must be an edge of $H$ contained in these classes which is monochromatic, which is not a valid colouring, and hence not a valid $Q$-colouring of $H$. We conclude that $t \leq \chi(H) \leq \chi_Q(H) \leq n$, and $\omega(H) \leq (r-1)^2$.
\end{proof}

\section{Q-colourings and gaps}

We now turn to the conditions for the existence/non-existence of gaps in the $Q$-spectrum of a hypergraph $H$.  We start with some definitions. These are similar to the definitions given in \cite{2010pattern}, but in our case, they are restricted to non-oriented $r$-uniform hypergraphs.  Our approach in this section incorporates and develops several arguments  from  \cite{2010pattern}, together with arguments specifically  developed to treat $\Sigma$-hypergraphs.  

Consider a partition $\sigma=(a_1,a_2,\ldots,a_s)$ of $r$ into $s$ parts.  We say that a partition $\pi$ is \emph{derived from $\sigma$ by reduction} if it was formed by choosing $a_i,a_j$ in $\sigma$ and replacing them by $a_i+a_j$.

The set of partitions of $r$ that are formed by repeatedly applying reduction starting from $\sigma$ is denoted by $RD(\sigma)$.  Similarly, for a set of partitions $Q \subseteq P(r)$, $RD(Q)$, the set of reduction-derived partitions is \[\bigcup _{\sigma \in Q} {RD(\sigma)}.\]

A set $Q \subseteq P(r)$ is called \emph{reduction closed} if $RD(Q)=Q$, and $RD(Q)$ is called the \emph{reduction closure of Q}.  It is clear that $RD(RD(Q))=RD(Q)$ and that $M$ is always in $RD(Q)$.

As an example, consider $r=6$ and $\sigma=(3,1,1,1)$.  The partitions derived from $\sigma$ are $(4,1,1)$ and $(3,2,1)$ - applying reduction to these two partitions give the partitions $(5,1)$ $(4,2)$ and $(3,3)$, which in turn give the partition $(6)=M$.  Hence $RD((3,1,1,1))=\{(3,1,1,1),(4,1,1),(3,2,1),(5,1),(4,2),(3,3),M\}$.

Now consider a partition $\sigma=(a_1,a_2,\ldots,a_s)$ of $r$ into $s$ parts.  We say that a partition $\pi$ is \emph{derived from $\sigma$ by expansion} if it was formed by choosing $a_j$ in $\sigma$ and replacing it by $a_j-1$ and $1$.

The set of partitions of $r$ that are formed by repeatedly applying expansion starting from $\sigma$ is denoted by $EX(\sigma)$.  Similarly, for a set of partitions $Q \subseteq P(r)$, $EX(Q)$, the set of expansion derived partitions is \[\bigcup _{\sigma \in Q} {EX(\sigma)}.\]

A set $Q \subseteq P(r)$ is called \emph{expansion closed} if $EX(Q)=Q$, and $EX(Q)$ is called the \emph{expansion closure of Q}.  It is clear that $EX(EX(Q))=EX(Q)$ and that $R$ is always in $EX(Q)$.

As an example, consider $r=6$ and $\sigma=(3,3)$.  The partition derived from $\sigma$ is $(3,2,1)$, and those derived from $(3,2,1)$ are $(2,2,1,1)$ and $(3,1,1,1)$.  These in turn give $(2,1,1,1,1)$ and finally $(1,1,1,1,1,1)=R$.  Hence $EX((3,3))=\{ (3,3),(3,2,1),(2,2,1,1),(3,1,1,1),(2,1,1,1,1),R\}$.

A set $Q$ of partitions of $r$ is called \emph{simply closed} if $Q$ contains $EX(M)=\{(r),(r-1,1), (r-2,1,1),\ldots, (2,1,1,1,\ldots,1),(1,1,\ldots,1)\}$, the expansion closed family which is obtained starting with the partition $M$.  

We call a set $Q \subseteq P(r)$ \emph{robust}  if it is reduction-closed, expansion-closed or simply-closed.

We can rephrase Theorem 20 in \cite{2010pattern} using our terminology as follows:

\begin{theorem} \label{theorem20}
There are no gaps in the $Q$-spectrum of any $r$-uniform hypergraph $H$ if and only if $Q \subseteq P(r)$ is robust.
\end{theorem}

We shall here prove the necessity of the condition.

\begin{proof}
\noindent \emph{Case 1 --- $Q$ is reduction-closed}

$\mbox{  \\}$

Let $H$ be an $r$-uniform hypergraph.  Clearly, $H$ is $1$-$Q$-colourable since $M \in Q$.  Now suppose $k$ is the maximum number of colours for which $H$ is $k$-$Q$-colourable.  Consider a $k$-$Q$-colouring of $H$, hence each edge has $pat(E) \in Q$.  Let us recolour the vertices coloured $k$ with the colour $k-1$.  Let us consider the possible effects of this:
\begin{enumerate}
\item{Consider the edges in which the colour $k$ did not appear.  Then the colouring pattern of these edges has not changed and hence is still a colour pattern in $Q$.}
\item{Consider edges in which the colour $k$ appeared, but not the colour $k-1$.  If we change the colour $k$ to $k-1$, the colour pattern of the edges remain the same.}
\item{Consider edges in which the colour $k$ appeared, as well as the colour $k-1$.  If there are $a_i$ vertices coloured $k$ and $a_j$ vertices coloured $k-1$, then the new edge has $a_i+a_j$ vertices coloured $k-1$, and this colour pattern is in $Q$ since $Q$ is reduction closed.}
\end{enumerate}

So  we have a valid $(k-1)$-$Q$-colouring of $H$.  This process can be repeated until we end up with one colour, that is a monochromatic colouring.  This is possible since $Q$ is reduction closed.  Therefore there are no gaps in the $Q$-spectrum of $H$.

$\mbox{  \\}$

\noindent \emph{Case 2 --- $Q$ is expansion-closed}

$\mbox{  \\}$

Let $H$ be an $r$-uniform hypergraph on $n$ vertices.  Clearly, $H$ is $n$-$Q$-colourable since $R \in Q$.  Now suppose the minimum colours for which a $Q$-colouring of $H$ exists is $k<n$, and consider a $k$-$Q$-colouring of $H$.  Let us choose a vertex $v$  of colour $x$ which appears at least twice ( there must be such a colour since $k < n$).  Let us change the colour of $v$ to a new colour $y$.  Hence there are now $k+1$ colours on the vertices of $H$.   Let us consider the possible effects of this:
\begin{enumerate}
\item{For edges which do not include the vertex $v$, the colour pattern has not changed.}
\item{If the edge includes the vertex $v$, but no other vertex of colour $x$, the new edge has the same colour pattern.}
\item{If the edge includes the vertex $v$ and other vertices of colour $x$, let the number of vertices of colour $x$ (including $v$) be $a_i$.  If we change the colour of $v$ to colour $y$, we now have $(a_i-1)$ vertices of colour $x$, and one vertex of colour $y$.  Since $Q$ is expansion-closed, this new colour pattern is in $Q$.}
\end{enumerate}

Hence we have a valid $(k+1)$-$Q$-colouring of $H$.  Again, this process can be repeated until we reach $n$ colours since $Q$ is expansion-closed.  Therefore there are no gaps in the $Q$-spectrum of $H$.

$\mbox{  \\}$

\noindent \emph{Case 3 --- $Q$ is simply-closed}

$\mbox{  \\}$

Let $H$ be an $r$-uniform hypergraph on $n$ vertices.  For each $k$, $1 \leq k \leq n$ we colour the vertices of $H$ as follows:  we take $k-1$ vertices say $v_1$ to $v_{k-1}$ and colour them with $k-1$ distinct colors say $1 ,\ldots ,k-1$.  We colour the remaining vertices all with the same colour $k$.  Any edge will have the colour pattern $(r-j,1,1,\ldots,1)$ for $0 \leq j \leq r$, so $H$ is $k$-$Q$-colourable for $1 \leq k \leq n$.

\end{proof}

Rather than dealing with any $r$-uniform hypergraph to prove the full converse (which can be done following the original proof modified to the $Q$-colouring language), we now prove a partial converse to Theorem \ref{theorem20} for $\Sigma$-hypergraphs, and then prove that hypergraphs of the form $H=H(n,4,q|\Sigma)$  have no gap in their $Q$-spectrum  where $Q = \{ ( 3,1) \}$ despite $Q$ being non-robust.

\begin{theorem} \label{necessity}
If either  $M \in Q$ or $R \in Q$, and $Q$ is not a robust family, then there exists a $\Sigma$-hypergraph $H$ with a gap in its $Q$-spectrum.
\end{theorem}

We shall prove this theorem in a series of three lemmas.

\begin{lemma}
Suppose $Q \subset P(r)$ and $M,R \in Q$, but $Q$ is not simply closed.  Then the $\Sigma$-hypergraph $H$ on $r^3$ vertices (and a hypergraph K on $r^2$ vertices)  has a gap in its $Q$-spectrum.
\end{lemma}

\begin{proof}
Consider the complete $r$-uniform hypergraph $K$ on $r^2$ vertices.  It is $1$-$Q$-colourable since $M \in Q$, and it is $r^2$-$Q$- colourable since $R \in Q$.  Now consider an $r$-colouring of $K$.  Then at least one colour, say colour 1, appears at least $r$ times, say on vertices $v_1,\ldots,v_r$.  Each other colour appears at least once, say colour $j$, for $2 \leq j \leq r$ appears on the vertex $u_j$.  Then the edge containing $k$ vertices chosen from $v_1, \ldots,v_r$, and $r-k$ vertices chosen from $u_2, \ldots, u_r$ has colouring type $(r-k,1,1,\ldots,1)$, $\forall k=0 \ldots r$.  But this means that $EX(M)$ must be contained in $Q$, a contradiction.

Now consider the $\Sigma$-hypergraphs $H=H(r^2,r,r \mid \Sigma=Q)$ and $H=H(r,r,r^2 \mid \Sigma=Q)$ --- the hypergraphs are $1$-colourable since $M \in Q$ and $r^3$-colourable since $R \in Q$., but is not $r$-colourable since:
\begin{enumerate}
\item{in $H=H(r^2,r,r \mid \Sigma=Q)$, the set of vertices obtained by choosing one vertex from each class has $r^2$ vertices and since $R \in \Sigma=Q$ it is a complete $r$-uniform hypergraph on $r^2$ vertices,so is not $r$-$Q$-colourable by the above argument.}
\item{in $H=H(r,r,r^2 \mid \Sigma=Q)$, , each class $V_i$ of $H$ is a complete $r$-uniform hypergraph on $r^2$ vertices since $M \in Q$,  and again, by the above argument is not $r$-$Q$-colourable.}
\end{enumerate}

\end{proof}

\begin{lemma}
Let $Q \subset P(r)$ such that $R \in Q$, $M \not \in Q$ and $Q$ is not expansion-closed.  Then the $\Sigma$-hypergraph $H=H(r,r,r^2 \mid \Sigma = Q)$ on $r^3$ vertices has a gap in its $Q$-spectrum.
\end{lemma}

\begin{proof}

Consider $H=H(r,r,r^2 \mid \Sigma=Q)$.  Then $H$ is $r$-$Q$-colourable using the CDMC, and $H$ is $r^3$-$Q$-colourable since $R \in Q$.  We claim that $H$ is not $(r+1)$-$Q$-colourable.  Suppose that $f$  is a proper $(r+1)$-$Q$-colouring of $H$.  We will show that this assumption contradicts the fact that $Q$ is not expansion-closed. 

Let $f_i$ be the colour which appears most in class $V_i$ --- it is clear that since $q =r^2$ and we use at most $r+1$ colours in each class, $f_i$ appears at least $r$ times in class $V_i$, for $1 \leq i \leq r$. So we may assume that, without loss of generality, the first $r$ vertices in $V_i$, $v_{i,j}$ for $j=1 \ldots r$, are such that $f(v_{i,j})=f(i)$.  

We first  show that $f(i) \not = f(j)$ for $1 \leq i,j \leq r$.  Suppose for contradiction that for some $i,j$, $f(i)=f(j)$.  Recall that $\Delta(Q)= \max \{ \Delta(\pi): \pi \in Q \}$ , and since $M \not \in Q$, $\Delta(Q) < r$.  Consider an edge $E \in E(H)$ with $type(E)=\sigma$ and  $\Delta(\sigma)=\Delta(Q)$, such that the $\Delta(\sigma)$ vertices are chosen among the first $r$ vertices in $V_i$, so that they all have colour $f(i)$, and such that another part of $\sigma$ is chosen from the first $r$ vertices of the class $V_j$.  Since $f(i)=f(j)$, this would imply that edge $E$ has at least $\Delta(Q)+1$ vertices of the same colour, contradicting the maximality of $\Delta(Q)$.

We will now show that the assumption that $f$ is a proper $(r+1)$-$Q$-colouring of $H$ and the fact that $Q$ is not expansion-closed are contradictory.  Since $Q$ is not expansion-closed, there is a partition $\sigma=(a_1,a_2,\ldots,a_s)$ such that for some $j \geq 2$, the derived partition $\sigma^*=(a_1,\ldots,a_j-1,\ldots,a_s,1)$ is not in $Q$.

Since there are $r$ classes in $H$, we may assume that $f(i)=i$ for $i=1 \ldots r$ and that the colour $r+1$ appears in class $V_1$, without loss of generality.  Let us also assume that $v_{1,r+1}$ has colour $(r+1)$.  

Let $E \in E(H)$ be an edge with $type(E)=\sigma$ such that the part $a_j$ is taken from the first $a_j <r$ vertices of $V_1$ and the remaining parts from other classes, with all parts taken from the first $r$ vertices of the respective class.  Thus $E$ is properly coloured.  Now in part $a_j$ let us replace the vertex $v_{1,1}$ with the vertex $v_{1,r+1}$ to create $E^*$ which is still a valid edge.  Clearly $pat(E^*)=\sigma^*$ which is not in $Q$, contradicting the assumption that $f$ is a proper $(r+1)$-$Q$-colouring of $H$.
\end{proof}

\begin{lemma}
Suppose $Q \subset P(r)$ such that $M \in Q$ but $R \not \in Q$, and $Q$ is not reduction-closed.  Then the $\Sigma$-hypergraph $H=H(r^2,r,r \mid \Sigma=Q)$ on $r^3$ vertices has a gap in its $Q$-spectrum.
\end{lemma}
\begin{proof}

Consider $H=H(r^2,r,r \mid \Sigma=Q)$.  Clearly $H$ is $1$-$Q$-colourable since $M \in Q$, and $H$ is $r^2$-$Q$-colourable using the CDMC.  We claim that $H$ is not $(r^2-1)$-$Q$-colourable.  Suppose on the contrary that $f$ is a proper $(r^2-1)$-$Q$-colouring of $H$.

We first show that each class must be monochromatic.  Assume without loss of generality that in class $V_1$, $1=f(v_{1,1}) \not = f(v_{1,2})=2$.  Since we are using $r^2-1$ colours and each class has $r$ vertices, there must be a vertex of another colour, say 3, not in $V_1$.  Let this vertex be $v_{2,1}$ in $V_2$.  Still $V_1$ and $V_2$ can have at most $2r$ colours, hence there is another vertex of colour say 4 not in $V_1 \cup V_2$.  Let this vertex be $v_{1,3} \in V_3$.  We can continue this process at least until class $V_r$, since $r-1$ classes can have at most $r(r-1) < r^2-1$ colours.  Hence we conclude that we can assume the vertices $v_{1,1},v_{1,2},v_{2,1},v_{3,1},\ldots,v_{r,1}$ are $r+1$ vertices of $r+1$ distinct colours.  Now let $\sigma \in \Sigma=Q$ be a partition such that $s(\sigma)=s(Q)$.  Since $R \not \in Q$, $s(Q) < r$ and hence $\Delta(\sigma) \geq 2$.  We choose an edge $E$ with $type(E)=\sigma$ such that the part $a_j \in \sigma$ for $j=1 \ldots s$ consists of the first $a_j$ vertices of class $V_j$.  Clearly the $pat(E)=\pi$ where $s(\pi) \geq s(\sigma)+1$, hence $\pi \not \in Q$.  Thus all classes must be monochromatic.

Since there are $r^2$ classes and we are using $r^2-1$ colours, there are two classes, say $V_1$ and $V_2$ which are monochromtic of the same colour, while all other classes are monochromatic with pairwise distinct colours. Since $Q$ is not reduction-closed, there is a parition $\pi=(a_1,\ldots,a_s) \in Q$ such that for some $1 \leq i,j \leq s$, the partition $\pi^*$ formed by replacing $a_i$ and $a_j$ by $a_i+a_j$ is not in $Q$.

Now consider the edge $E^*$ with $type(E)=\pi \in \Sigma=Q$ such that the parts $a_i$ and $a_j$ are chosen from classes $V_1$ and $V_2$, while the other parts are chosen from any of the remaining classes.  Clearly $pat(E)=\pi^* \not \in Q$, contradicting the assumption that $f$ was a proper $(r^2-1)$-$Q$-colouring.

\end{proof}

Combining these three lemmas, we have a proof of Theorem \ref{necessity}.

We now consider the case $Q= \{ ( 3,1) \}$. Clearly $M$ and $R$ are not in $Q$, and $Q$ is not robust.  By Theorem \ref{theorem20} there exists a hypergraph $H$ that has a gap in its $Q$-spectrum.
 
We will show that this is not a $\Sigma$-hypergraph, hence only a partial converse to Theorem \ref{theorem20} can holds for $\Sigma$-hypergraphs.  Yet we give a concrete example  of a small 4-uniform hypergraph that has a gap in its $Q$-spectrum.  

We have checked that
\begin{enumerate}
\item{Theorem \ref{theorem20}  holds for $\Sigma$-hypergraphs for $r = 3$  namely if $Q$  is not robust then there is $H(n,3,q|\Sigma)$ with a gap in its $Q$-spectrum.}
\item{Theorem \ref{theorem20} holds for $\Sigma$-hypergraphs for $r = 4$ unless $Q = \{(3,1)\}$,  namely if $Q$ is not robust and not $\{ ( 3,1)\}$ then there is $H(n,4,q|\Sigma)$ with a gap in its  $Q$-spectrum.}
\end{enumerate}

The next Proposition  covers the case $Q = \{ ( 3,1) \}$.

\begin{proposition}
Consider $H(n,4,q | \Sigma)$ where $\Sigma \subseteq P(4)$, and $Q=\{(3,1)\}$.  Then $H$ does not have a gap in its $Q$-spectrum for $n \geq s(\Sigma)$, $q \geq \Delta(\Sigma)$.
\end{proposition}
\begin{proof}

We first consider  the case where $\Sigma$ consists of a single partition of  four  (there are 5 such partitions).  The conditions $n \geq s(\Sigma)$, $q \geq \Delta(\Sigma)$ are included so that all edge types in $\Sigma$ can be realised.
\begin{enumerate}
\item{Consider $H(n,4,q | \sigma=(3,1))$ where $n \geq 2$, $q \geq 3$.  Then $H$ is $n$-$Q$-colourable using the CDMC.  In general, no class can contain 3 colours since in that case we can choose three vertices of distinct colours to give an edge containing three colours.  So we may assume that each class contains at most two colours.

Consider a $k$-colouring, where $k < n$.  Then at least one colour, say the colour 1, appears in two distinct classes, say $V_1$ and $V_2$, without loss of generality.   If in one of these classes, say $V_1$, there are three vertices of colour 1, then we can choose a monochromatic edge from $V_1$ and $V_2$.  If in say $V_1$  there are two vertices of the same colour but not colour 1, then we can choose these two vertices and the vertex of colour 1 from $V_1$, plus the vertex of colour 1 in $V_2$, to give an edge with $pat(E)=(2,2)$.  So there can only be one vertex of another colour, say 2, in $ V_1$.  But then all vertices in $V_2$ must be of colour 1, because otherwise, if in $V_2$ there is a vertex of colour 2, we can choose an edge which includes two vertices of colour 1 and two vertices of colour 2, and if there is a vertex of another colour other than 1 and 2 in $V_2$, then we can choose an edge to include three colours.  But if all the vertices in $V_2$ are of colour 1, and there is a vertex in $V_1$ of colour 1, we can choose a monochromatic edge.  Hence it is not possible to colour $H$ with less than $n$ colours.  

Now consider  a $k$-colouring, where $k > n \geq 2$. Then at least one class contains two vertices of distinct colour, say in $V_1$ there is a vertex of colour 1 and a vertex of colour 2.  But since there are at least three colours, there is a vertex of another colour, say colour 3 in some other class in $H$.   But then we can choose an edge to include the two vertices of colours 1 and 2 in $V_1$, and this vertex of colour 3, giving an edge which includes 3 colours.  Hence a $Q$-colouring with more than $n$ colours is not possible.  Therefore $Spec_{Q}(H)=\{n\}$.}
\item{Consider $H(n,4,q | \sigma=(2,2))$ where $n \geq 2$, $q \geq 2$.  If $n=q=2$ we can colour one class monochromatically, say with colour 1, and in the other class, we colour a vertex with colour 1 and one with colour 2, and this is the only valid $Q$-colouring of this graph, and hence its $Q$-spectrum is not broken.  If $q>2$, then if $V_1$ is monochromatic of say colour 1, then the vertices in $V_2$ must be coloured in such a way that we always choose a vertex of colour 1 and one of another colour.  But this is impossible if $q>3$.  

Hence consider $q =2$ and $n>2$.  If we colour the first 2 classes as described above, then if in the third class, both vertices are of colour 1, we can choose a monochromatic edge from classes $V_1$ and $V_3$, while if both vertices are coloured 2, we can choose and edge $E$ from classes $V_1$ and $V_3$ such that $pat(E)=(2,2)$.  If both colours 1 and 2 appear in $V_3$, then we can choose an edge with colour pattern $(2,2)$ from classes $V_2$ and $V_3$.  Finally, if a third colour appears in $V_3$, then we can choose an edge which includes three colours from $V_2$ and $V_3$.  So again $H$ is not colourable, and hence its $Q$-spectrum is empty.}

\item{Consider $H(n,4,q | \sigma=(2,1,1))$ where $n \geq 3$, $q \geq 2$.  Consider two vertices of the same colour, say colour 1, in a class, say $V_1$.  Hence without loss of generality, we may  assume in $V_2$ there is a vertex of colour 1 and in $V_3$ a vertex of colour 2.   This forces $V_2$ to be monochromatic of colour 1, otherwise the two vertices from $V_1$ and the vertex from $V_3$ together with a vertex from $V_2$ that is not of colour 1 will give either an edge $E$ with $pat(E )= (2,2)$ or $pat(E ) = (2,1,1)$.
 
If  $V_3$ is not monochromatic of colour 2 then either there is a vertex of colour 1 in $V_3$  and we have a monochromatic edge choosing the two vertices from $V_1$, any vertex from $V_2$ and this vertex from $V_3$, or there is an edge that includes three colours  choosing the two vertices from $V_3$ and any vertex of $V_2$ and a vertex of colour 1 from $V_1$.  Hence $V_3$ must be  monochromatic of colour 2 and choosing any two vertices from $V_3$, one from $V_2$ and a vertex of colour 1 from $V_1$, we get an edge with color pattern $(2,2)$.

So assume that no two vertices in $V_1$ have the same colour.  Then if we choose two vertices from this class, say one of colour 1 and one of colour 2 without loss of generality, then the remaining classes must be monochromatic, either all of colour 1 or all of colour 2.  Let us say they are all of colour 1.  But then we can choose a monochromatic edge by choosing a vertex of colour 1 from $V_1$, a vertex from $V_2$ and two vertices from $V_3$, all of colour 1.  Hence $H$ is not $Q$-colourable.}

\item{Consider $H(n,4,q | \sigma=(1,1,1,1))$ where $n \geq 4$, $q \geq 1$.  If $n=4$ and $q \geq 1$, then without loss of generality we can colour $V_1$, $V_2$ and $V_3$ monochromatically using colour 1, while the vertices in $V_4$ can be coloured using colours different from 1, that is using one colour, say colour 2, up to $q$ colours.  Hence $H(4,4,1 |\sigma=(1,1,1,1))$ has $Q$-spectrum $\{2,3,\ldots,q+1\}$.  Now consider a $k$-colouring of $H$ for $k > q+1$.  Let us choose an edge $E=\{v_1,v_2,v_3,v_4\}$ such that $v_i$ is chosen from $V_i$ for $i=1 \ldots 4$, so that $pat(E)=(3,1)$, and without loss of generality, we assume that $v_1$, $v_2$ and $v_3$ ave the same colour, say colour 1, and $v_4$ is of a different colour.  Clearly, all vertices in $V_1$,$V_2$ and $V_3$ must be of colour 1, otherwise we can choose an edge with colour pattern $(2,2)$ or $(2,1,1)$ which would make the colouring invalid.  Hence only class $V_4$ can contain vertices of different colours, and since there are $q$ vertices in a class, the maximum number of colours which can be used in a valid $Q$-colouring is $q+1$.

If $n>4$, then we must choose a vertex from each class such that three have the same colour, and one has different colour and if $q>1$, this is clearly not possible.}
\item{Consider $H(n,4,q | \sigma=(4))$ where $n \geq 1$, $q \geq 4$.  Then we must choose four vertices from one class such that three are of the same colour, and one is of a different colour.  If $q>4$, this is not possible.  For $q=4$, we can colour each class to include 3 vertices of one colour, and one vertex of a different colour.  This means that the $Q$-spectrum of such a graph is $\{2,3,\ldots,2n\}$, and it contains no gaps.}
\end{enumerate}

Hence if $n \geq s(\sigma)$, $q \geq \Delta(\sigma)$, and  $H(n,4,q | \sigma)$ where $\sigma \in P(4)$, there is no gap in the $Q$-spectrum of $H$, for $Q=\{(3,1)\}$.

Having done the above five cases we use the following simple observation.   Let $\Sigma= \Sigma_1 \cup \Sigma_2$  where $\Sigma_1 \cap \Sigma_2  = \emptyset$  then \[Spec_Q(H(n,r,q|\Sigma)) \subseteq Spec_Q(H(n,r,q|\Sigma_1)) \cap Spec_Q(H(n,r,q|\Sigma_2)),\] and hence one can easily verify that for any $\Sigma$ consisting of at least two partitions of 4, $Spec_Q( H(n,r,q|\Sigma))$ is either empty or contains exactly one value, hence  it is not broken.
 
\end{proof}

Although no $\Sigma$-hypergaph has a gap in its $(3,1)$-spectrum, we now give an example of a 4-uniform hypergraph with a broken $(3,1)$-spectrum.  

Consider the $4$-uniform hypergraph $H$ having sixteen vertices.  Suppose they are grouped into cells each containing two vertices, and the cells are arranged in a $4 \times 2$ grid. The cell in row $i$ column $j$ will be denoted by $C_{ij}$, $1 \leq i \leq 4$ and $1 \leq j \leq 2$.  The vertices in cells $C_{i1}$ and $C_{i2}$ are said to be in the $i^{th}$ row and the vertices in cells $C_{1j}$ to $C_{4j}$  are said to be in column $j$.  All 4-subsets $K$ of four vertices form an edge if the sizes of the non-empty intersections of $K$ with the rows, and also the sizes of the non-empty intersections of $K$ with the columns form the partition $(3,1)$.  Figure \ref{hyp31} shows an example with two edges in this hypergraph.

\begin{figure}[h!]
\centering
\includegraphics{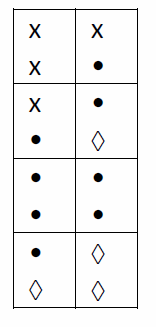}
\caption{Two edges are shown with the vertices in one edge marked with $\lozenge$ while in the other edge the vertices are marked as $\times$ } \label{hyp31}
\end{figure}

We show that $H$ is $2$-$Q$-colourable, $4$-$Q$-colourable for $Q=\{(3,1)\}$, but not $3$-colourable.  Consider a $2$-colouring using colours 1 and 2 such that   the vertices in colourn 1 receive colour 1, and the vertices in column 2 receive colour 2.  It is easy to see that an edge $E \in E(H)$ is such that $pat(E)=(3,1)$.  Consider a 4-colouring of $H$ such that the vertices in row $i$ receive colour $i$ for $i=1 \ldots 4$.  Again it is easy to see that $pat(E)=(3,1)$, $\forall E \in E(H)$.

We now show that $H$ is not $3$-$Q$-colourable.  Suppose $H$ is coloured using three colours 1,2 and 3.  Suppose first that the vertices in a cell receive different  colours.  We may assume without loss of generality that the vertices in $C_{11}$ receive colours 1 and 2.  There must be a vertex $v$ which receives colour 3.  If $v$ is in column 1, then we can choose an edge which contains the three vertices of different colours, and hence the coluring is not valid.  Therefore all vertices in column 1 must have colour 1 or 2.

Suppose $v$ appears in column 2.  If $v$ is in row 1, we can choose an edge which includes the vertices in row 1 coloured 1,2 and 3, which is again  invalid.  Therefore we can assume $v$ is not in row 1, and we assume without loss of genreality that $v$ is in $C_{22}$.  SO we can choose an edge $K$ as follows: $K$ contains $v$ and the two vertices in $C_{21}$.  If these two vertices have colours 1 and 2 (recall that they cannot have colour 3), then already the colouring is invalid.  If they are both of the same colour, say colour 1, we can choose the final vertex from $C_{11}$ with column 2, (and similarly if both have colour 2, we can choose a vertex of colour 1 from $C_{11}$), again showing that the colouring is not valid.

So we may conclude that the vertices in one cell must have the same colour.  Without loss of generality, suppose the vertices in $C_{11}$ receive colour 1.  If there is another vertex of colour 1 in the first coloumn, say in row 2, then the vertices in $C_{12}$ must receive another colour, say colour 2, otherwise we can choose a monochromatic edge.  But there must be a vertex $v$ coloured 3.  If $v$ is in the first coloumn, then letting edge $K$ contain $v$, the two vertices in $C_{11}$ and a vertex coloured 2 from $C_{12}$ gives an edge with three colours.  Similarly, if $v$ is in column 2 we have the same situation.  Hence colour 1 cannot appear on the other cells in column 1.  By the same argument, no colour can appear in two cells in the same coloumn.  But then for cells $C_{21}$, $C_{31}$ and $C_{41}$ only colours 2 and 3 are available(recall that the vertices in a cell must receive the same colour), and therefore the vertices in at least two cells must all receive the same colour, giving us the same contradiction.

Hence there is no valid $3$-$Q$-colouring of $H$, and $Spec_Q(H)$ contains a gap.

\section{Conclusion}

In this paper we have tried to unify and clarify a number of hypergraph colouring problems using the language of $Q$-colourings and $\Sigma$-hypergraphs.  We have defined tight colourings, which is a generalisation of unique colourability, and we have shown that for all $Q$ not containing the partitions $M$ and $R$ there exist $\Sigma$-hypergraphs which are tightly Q-colourable.  We have also shown that finding the clique number of $\Sigma$-hypergraphs, which is, in general,  an $NP$-complete problem, can be done in $O(1)$ time and that there exist $\Sigma$-hypergraphs with arbitrarily high chromatic, and hence $Q$-chromatic, number but bounded clique number.

Finally we have considered the result in \cite{2010pattern}, which characterises those $Q$ for which there exist hypergraphs with a gap in their $Q$-spectrum.  In \cite{2010pattern}, this characterisation is given for oriented hypergraphs and the result for hypergraphs is given as a corollary.  Here we give a short direct proof of the necessity of their condition for $\Sigma$-hypergraphs using the language of $Q$-colourings.  We also prove a partial converse of their result for $\Sigma$-hypergraphs and show that, for $Q=\{(3,1)\}$, which is known to give hypergraphs with gaps, there is no $\Sigma$-hypergraph with a gap in its $Q$-spectrum.  However, we present a hypergraph with a gap in its $Q$-spectrum for $Q=\{(3,1)\}$.  This hypergraph can be viewed as a further generalisation of $\sigma$-hypergraphs, which we intend to study further.

\bibliographystyle{plain}
\bibliography{qcol}

\end{document}